\documentclass{amsart}
\usepackage{amssymb,latexsym}
\usepackage{graphicx}
\usepackage{amsfonts}
\usepackage{epstopdf}
\usepackage{float}

\newtheorem{theorem}{Theorem}
\newtheorem{lemma}{Lemma}
\newtheorem{definition}{Definition}
\newtheorem{corollary}{Corollary}

\newtheorem{question}{Question}
\newtheorem{example}{Example}
\newtheorem{conjecture}{Conjecture}

\newtheorem{proposition}{Proposition}
\newtheorem{computer search}{Computer Search}
\begin{document}

\title{Slide Statistics And Financial Returns}
\author{William J. Ralph}

\address{Mathematics Department\\
Brock University\\
St. Catharines, Ontario\\
Canada L2S 3A1\\
\textnormal{Email: bralph@brocku.ca}\\
\textnormal{Phone: (905) 688-5550 x3804}}

\date{March 18, 2015}
\begin{abstract}

We present a new approach to financial returns based on an infinite family of statistics called \textit{slide statistics} that we introduce. The evidence these statistics provide suggests that certain distributions such as the stable distributions are not good models for the financial returns from various securities or indexes like the S$\&$P $500$ and the Dow Jones. Formally, we associate with any finite subset of a metric space an infinite sequence of scale invariant numbers $\rho_1,\rho_2,\dots$ derived from a variant of differential entropy called the genial entropy. We give explicit formulas for $\rho_1$ and $\rho_2$ that are easily evaluated by a computer and make this theory particularly suitable for applications. As statistics for point processes, these numbers often appear to converge in simulations and we give examples where $1/\rho_1$ converges to the Hausdorff dimension and we prove that $\rho_1 \geq 0$. For a uniform random variable $X$ on $[a,b]^n$, the evidence from simulations suggests that $\rho_1(X) =1/n$ and $\rho_2(X) =-\pi^2/(6n^2)$ which yields new tests for spatial randomness. The slide statistics describe continuous random variables in an entirely new way.  For example, if $Z$ is any normal variable then simulations suggest that $\rho_1(Z) =4/\pi$ and $\rho_2(Z) =-1$ which provides new goodness of fit tests for normality. 
\end{abstract}


\maketitle
\section{Introduction}\label{S:intro}
We develop new entropy based statistics $\rho_1,\rho_2,\dots$ called \textit{slide statistics}  which can be computed from any sample data in a metric space. As an application, we use these statistics to test whether financial returns are independent observations from a particular distribution. For example, Figure~\ref{fig:Fig1} shows plots of  $\rho_1$ against $n$ for data $T_n$ consisting of n-tuples of consecutive returns regarded as a subset of the metric space $R^n$ with the usual metric. As can be seen, the lower curve corresponding to the S$\&$P $500$ is very different from the ones obtained for either the Normal or Laplace distributions. Any potential model for the returns of the  S$\&$P $500$ must be able simulate the $\rho_1$ curve in Figure~\ref{fig:Fig1} which is a new requirement that is apparently difficult to meet.  The statistics $\rho_i$ contain information about the interplay between the dimension of a space and distributions on that space and provide a new way of describing continuous random variables and stochastic processes in general. The $\rho_i$  will take some time to define but we will provide explicit formulas for slide statistics $\rho_1$  and $\rho_2$  that are easily computed for any sample in a metric space. 

\begin{figure}[!h]
\centering
\includegraphics[width=4in]{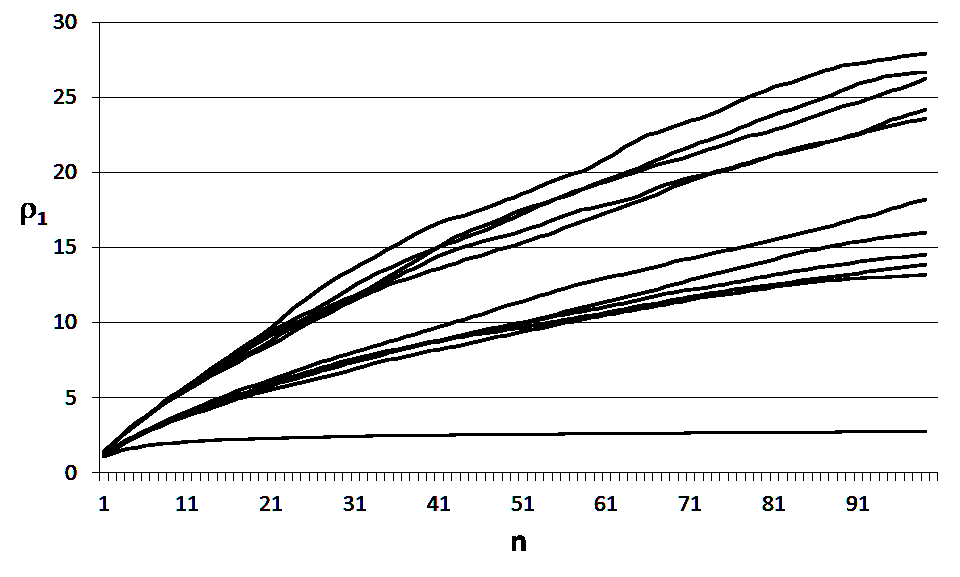}
\caption{The lower graph is the plot of  $\rho_1(T_n)$ against $n$ for the S$\&$P $500$ for the $5000$ trading days ending December 31, 2014.  The middle cluster of five graphs are the corresponding plots obtained from five different simulated sequences of returns with the Laplace distribution.  The upper cluster of five graphs are the corresponding plots obtained from five different simulated sequences of normally distributed returns.  }
\label{fig:Fig1}
\end{figure}

Our applications will focus on financial returns but many other important processes in science and mathematics also yield data in the form of a set of points in a metric space that must be quantified and interpreted. Fields like fractal analysis have developed in order to obtain dimensional information from a wide range of real world data using quantities like the Hausdorff, information and correlation dimensions ~\cite{F,H}.  One of the problems with these traditional measures of dimension is that they tend to concentrate on the geometric properties of a set while ignoring the statistical origins of the data.  We will give examples where $\rho_1(U)$ appears to converge to the reciprocal of the Hausdorff dimension when $U$ is taken to be a a larger and larger random sample from a fractal.

It is not clear in general how to assign a dimension to a point process but our theory and simulations suggest a purely statistical way to identify a class of them  for which it does make sense to assign a dimension.   Specifically, we think the dimension of a point process $P$ should be defined to be $d$ provided  $\rho_1(P)=1/d$ and $\rho_n(P)=(-1)^{n+1}(n-1)!(n-1)\zeta(n)/d^n$ for $n>1$ in which case we say the process is tangible with dimension $d$. In the case of a tangible process, the values of $\rho_n(P)$ are entirely determined by the number $d$ but this relationship does not hold in general. A possible example of a tangible process is the random generation of  points $(x_1,x_2,\dots,x_d)$ in $[a,b]^d$, where the $x_i$ are chosen independently and uniformly at random from $[a,b]$, for which we provide evidence that $\rho_1(P) =1/d$ and $\rho_2(P) =-\pi^2/(6d^2)$. In particular, $\rho_1$ and $\rho_2$ are new spatial statistics for testing hypotheses concerning the distribution of points in $[a,b]^d$. 

Although we will sometimes use fractals as examples, these statistics can be applied to any random variable with values in a metric space. In particular, the slide statistics appear to describe continuous random variables $X$ in an entirely new way that has nothing to do with mean and standard deviation since  $\rho_n(aX+b)=\rho_n(X)$ for $a\neq 0$ .  For example, the first two slide statistics for any normal variable $Z$ appear to converge to $\rho_1(Z) =4/\pi$ and $\rho_2(Z) =-1$ which provide the basis for a new goodness of fit test for normality. For any exponential random variable $W$, simulation gives that $\rho_1(W) \approx 1.4590$ and $\rho_2(W) \approx-0.7333$ and we conjecture from these and other simulations that $\rho_1(X) \geq 1$ for any continuous real-valued random variable $X$. If we think of $1/\rho_1(X)$ as the \textit{dimension} of the random variable $X$, then this last conjecture says the uniform distribution on $[a,b]$ has the maximum dimension of any continuous real-valued random variable. The second slide statistic $\rho_2$ is negative for the examples given so far but $\rho_2$ is positive for the Cauchy distribution and is often found to be positive for the log daily returns of the Standard and Poor 500 Index which is yet another demonstration of the non-normality of these returns.  

The construction of $\rho_i(U)$ requires several ideas working in concert and can be summarized as follows. In Section~\ref{S:2}, we introduce a variant of the differential entropy called the genial entropy which is the starting point for defining the slide statistics.  Unlike the differential entropy, the genial entropy is scale invariant and in Section~\ref{S:inequality} we prove it is never negative which will give a new lower bound for the differential entropy.  Given a finite set of distinct points $U$ in a metric space, we find the distance from each point to its nearest neighbour and arrange these distances in descending order so $ d_1 \geq d_2 \geq \dots \geq d_n > 0$. Let $f(x)$ be the function on $[0,1)$ whose value on $[(i-1)/n, i/n)$ is $d_i$. Let $A(t)$ be the area under $(f(x))^t$ and let $\sigma(t)$ denote the genial entropy of the density $\frac{(f(x))^t}{A(t)}$ which turns out to be $0$ at $t=0$. We then define $\rho_n(U)$ to be the $n$th derivative from the right at $0$ of the function $\sigma(t)$ which is developed in Section~\ref{S:calculus}.  

In Section~\ref{S:calculus}, we derive an explicit formula for $\rho_1(U)$ and state a conjectured formula for $\rho_2(U)$. The fact that $\rho_1$ and  $\rho_2$ are easily evaluated using a computer makes them particularly suitable for practical applications. The results obtained from the simulation of $\rho_1$ and $\rho_2$ in a variety of contexts are summarized in Section~\ref{S:slide} and demonstrate the interplay between dimension and distribution that is captured by these statistics. As applications in Section~\ref{S:returns}, we consider a variety of possible distributions as models for financial returns and use the slide statistics to illustrate how poorly these models fit the empirical data.

\section{The genial Entropy}\label{S:2}
Our starting point for the development of the slide statistics is the differential entropy $-\int fln(f)dx$ of a density $f$ which is well known in statistics and information theory ~\cite{CO}. Our focus will be on a variant of the differential entropy called the genial entropy or g-entropy which will be described in Definition~\ref{D:2.02}. In the simplest case of a probability density $f$ that has an inverse function such as $-ln(x)$ on $(0,1]$, the genial entropy is just the sum of the differential entropies of $f$ and $f^{-1}$.  The next proposition shows how this sum can be written in a form that makes sense for densities that may not have inverses.

\begin{proposition}\label{P:2.01} Suppose $f$ is a continuous function on $[0,b]$ with the properties that the derivative of $f$ exists and is negative on $(0,b)$, $f(b)=0$ and $\int_0^b fdx=1$. Also assume the differential entropies of $f$ and  $f^{-1}$ both exist. Then the sum of their differential entropies is given by $ -\int_0^b fln(f)dx-\int_{0}^{f(0)} f^{-1}ln(f^{-1})dy =  -1-\int_0^b fln(xf)dx$. 
\end{proposition}

\begin{proof} 
Substitute $y=f(x)$ into  the integral $-\int_{0}^{f(0)} f^{-1}ln(f^{-1})dy$ to get $-\int_{b}^{0} xln(x)f'(x)dx$ which equals $\int_{0}^{b} xln(x)f'(x)dx$. After integrating by parts this last integral becomes $ -1-\int_0^b fln(x)dx$ and the result follows.
\end{proof}
The genial entropy will only be defined for densities of the following special form.
\begin{definition} \label{D:2.01}  A corner density is a function $f:I \rightarrow [0,\infty)$ where $I \cup \{0\}$ is a connected interval contained in $[0,\infty)$, $f$ is monotone decreasing  and $\int_I fdx=1$. 
\end{definition}
Here is our definition of genial entropy which is motivated by the conclusion of Proposition~\ref{P:2.01}.
\begin{definition} \label{D:2.02} Let $f:I \rightarrow [0,\infty)$ be a corner density. The genial entropy or g-entropy of $f$ is defined by $G(f)= -1-\int_I fln(xf)dx$ when this integral exists, with the usual convention that $0 \ln (0) = 0$. 
\end{definition}
If $f$ satisfies the conditions of Proposition~\ref{P:2.01}, then $G(f) = G(f^{-1})$ so in particular $-\ln(x)$ and $e^{-x}$ must have the same genial entropy which happens to be Euler's constant as shown in Table~\ref{T:2.01}. In Section~\ref{S:inequality}, we will prove the genial entropy is always nonnegative.

\begin{table}[!htb]{The genial entropy of some corner densities.}
\begin{center}

\begin{tabular}{|c|c|c|}
 \hline
 \textit{Density} & \textit{Domain} &   \textit{Genial Entropy}  \\
\hline
$1/b$ & $[0,b]$ & $0$   \\
\hline
$-\ln(x)$  & $(0,1]$ & $\gamma$ \\
\hline
$e^{-x}$  & $[0,\infty)$ & $\gamma$   \\
\hline
$a/(x^{1-a})$ for $a \in (0,1)$  & $(0,1]$ & $-\ln(a)$  \\
\hline
$2e^{-x^2}/\sqrt{\pi}$  & $[0,\infty)$ & $(-1+\gamma+\ln(\pi))/2 $   \\
\hline
$2/(\pi(1+x^2))$  & $[0,\infty)$ & $-1+\ln(2)+\ln(\pi) $  \\
\hline

\end{tabular}
\caption{ }\label{T:2.01}
\end{center}
\end{table}
Unlike the differential entropy, the genial entropy is invariant under changes of scale. 

\begin{proposition}\label{P:2.02}  Let $f:I \rightarrow [0,\infty)$ be any corner density.  Then for any $\lambda >0$ the function $h:\lambda I \rightarrow [0,\infty)$ defined by $h(z)= \frac{1}{\lambda} f(\frac{z}{\lambda})$ for $z \in \lambda I$ is a corner density with the same genial entropy as $f$.
\end{proposition}

\begin{proof} 
$G(h) = -1-\int_{\lambda I} hln(zh)dz$ $=$ $-1-\int_{\lambda I} \frac{1}{\lambda} f(\frac{z}{\lambda})ln(z \frac{1}{\lambda} f(\frac{z}{\lambda}))dz$. After substituting $z=\lambda x$, this last integral becomes $-1-\int_I fln(xf)dx = G(f)$.
\end{proof}

In Section~\ref{S:slide}, we will consider data consisting of a finite set of distinct points in a metric space.  The first steps in computing the slide statistics will be to find the distances $d_i$ from each point to its nearest neighbour, arrange them in decreasing order and use them to construct the corner density $f_{D^*}$ discussed in the next definition and theorem.   

\begin{definition}\label{D:2.03} Let $D$ be any sequence $d_1 , d_2 , \dots , d_n $ with $ d_1 \geq d_2 \geq \dots \geq d_n > 0$ and let $D^*$ be the sequence $d_1/\mu , d_2/\mu , \dots , d_n/\mu $ where $\mu$ is the mean of the $d_i$.

\begin{enumerate}
\item Define $f_D:[0,1)\rightarrow[0,\infty)$ to have the value $d_i$ on the interval $[(i-1)/n, i/n)$. In particular, $f_{D^*}$ is the corner density whose value on $[(i-1)/n, i/n)$ is $d_i/\mu$.  
\item Define $L_{D}:[0,\infty)\rightarrow[0,1]$ by $L_{D}(y)= m/n$ where $m$ is the number of elements of $D$ that are less than or equal to $y$. In other words, $L_{D}$ is the restriction of the empirical cumulative distribution function for the data in $D$ to the interval $[0,\infty)$. 
\end{enumerate}
\end{definition}

The next theorem shows that the genial entropy of  $f_{D^*}$ can be computed using the empirical cumulative distribution function.

\begin{theorem}\label{T:2.03} Let $D$ be any sequence $d_1 , d_2 , \dots , d_n $ with $ d_1 \geq d_2 \geq \dots \geq d_n > 0$. Then $f_{D^*}$ and $1-L_{D^*}$ are corner densities with the same genial entropy.
\end{theorem}
\begin{proof} 
The sequence $D$ may contain repetitions, so assume that the discontinuities of $f_{D^*}$ in $(0,1)$ occur at $t_i$ with $t_1 < t_2 \dots < t_{m-1}$ and let $t_0=0$ and $t_m=1$. Suppose the value of $f_{D^*}$ on $[t_{i-1},t_i)$ is $e_i$ for $i=1,2,\dots,m$ and let $e_{m+1}=0$ so $ e_1 > e_2 > \dots > e_m > e_{m+1}=0$. We can then describe $1-L_{D^*}$ as the function that takes the value $t_i$ on $[e_{i+1},e_i)$ for $i=1,2,\dots,m$ and the value $0$ on $[e_1,\infty)$. Since $\int_0^1 f_{D^*}(x)dx=\sum_{i=1}^n (d_i/\mu )(i/n-(i-1)/n) =1$, we must also have that $\sum_{i=1}^m e_i(t_i-t_{i-1})=1$. Rearranging this last sum gives $\sum_{i=1}^m t_i(e_i-e_{i+1})=1$ so $\int_0^\infty 1-L_{D^*}(y)dy=1$ and $1-L_{D^*}$ is a corner density.
To see that  $f_{D^*}$ and $1-L_{D^*}$ have the same genial entropy, we use $t_0=0$, $e_{m+1}=0$ and the convention $0 \ln (0) = 0$ to obtain:
\begin{equation}
\begin{split}
-1-\int_0^1 f_{D^*}(x)ln(xf_{D^*}(x))dx
& =-1-\sum_{i=1}^m (\int_{t_{i-1}}^{t_i} e_i ln(x e_i)dx)\\
& =\sum_{i=1}^m ( t_{i-1} e_i  ln(t_{i-1} e_i )-t_i e_i  ln(t_i e_i ))\\
& =\sum_{i=1}^m ( e_{i+1} t_i ln(e_{i+1} t_i))-e_i t_i ln(e_i t_i) \\
& =-1-\sum_{i=1}^m (\int_{e_{i+1}}^{e_i} t_i ln(y t_i)dy)\\
& =-1-\int_0^{\infty} (1-L_{D^*}(y))ln(y(1-L_{D^*}(y)))dy
\end{split}
\end{equation}
\end{proof}

In view of Proposition~\ref{P:2.02}, we could have defined $f_{D^*}$ on any interval $[0,b)$ and obtained a function with the same genial entropy.    The interval $[0,1)$ was chosen in particular to insure the relationship between the genial entropies of  $f_{D^*}$ and $1-L_{D^*}$ stated in part (2) of Theorem~\ref{T:2.03}. We now show that the genial entropy is never negative.

\section{The genial entropy Inequality}\label{S:inequality}
Our goal in this section is to prove that the genial entropy of a corner density can never be negative which we will use to prove that the first slide statistic $\rho_1 \geq 0$. The idea is to first prove the necessary inequalities for step functions and then use the fact that monotone functions can be uniformly approximated by step functions.  We begin with an inequality for step functions.

\begin{lemma}\label{L:3.01} Suppose that $y_1 \geq y_2 \geq \dots \geq y_n \geq 0$ and $ 0 =t_0 \leq t_1 \leq t_2 \leq \dots \leq t_n$. Let $A = \sum_{i=1}^n y_i (t_i-t_{i-1})$. Then  $\sum_{i=1}^n (t_{i-1} y_i ln(t_{i-1} y_i)-t_i y_i ln(t_i y_i)) \geq -AlnA $
\end{lemma} 
\begin{proof} Given vectors $a,b \in R^n$, recall from ~\cite{MO} that $a$ majorizes $b$ provided \\ $\sum_{i=1}^k a_{(i)} \geq \sum_{i=1}^k b_{(i)}$ for all $k = 1 \dots (n-1)$ with equality for $k=n$ and where the components of $a$ and $b$ have been sorted in descending order so $a_{(1)} \geq a_{(2)} \geq \dots \geq a_{(n)}$ and $b_{(1)} \geq b_{(2)} \geq \dots \geq b_{(n)}$.  Alternatively, $a$ majorizes $b$ if vector $b$ can be obtained from vector $a$ by a sequence of "transfers" that allow us to change a vector $a = (a_1,a_2, \dots,a_i, \dots,a_j, \dots, a_n)$ into a vector $b = (a_1,a_2, \dots,a_i+\Delta, \dots,a_j -\Delta, \dots, a_n)$ provided $a_i \leq a_j$ and $\Delta \leq a_j-a_i$. 

We now show that the vector $(A, y_2 t_1, y_3 t_2, \dots , y_n t_{n-1})$ majorizes \\ $(y_1 t_1, y_2 t_2, \dots , y_n t_n)$ so the required inequality follows immediately from Karamata's Inequality ~\cite{MO} and the convexity of $x \ln x $ on $[0,\infty)$.  Since $A = y_n (t_n-t_{n-1}) +\sum_{i=1}^{n-1} y_i (t_i-t_{i-1})$ and  $ \sum_{i=1}^{n-1} y_i (t_i-t_{i-1})  \geq y_n\sum_{i=1}^{n-1}  (t_i-t_{i-1}) = y_n t_{n-1}$, we can transfer an amount $\Delta =y_n (t_n-t_{n-1})$ from the first to the last entry of $V_1=(A, y_2 t_1, y_3 t_2, \dots , y_n t_{n-1})$ which means that $V_1$ majorizes $V_2$ $=$  $(A-\Delta, y_2 t_1, y_3 t_2, \dots , y_n t_{n-1} +\Delta)$ $ =$ $ (\sum_{i=1}^{n-1} y_i (t_i-t_{i-1}), y_2 t_1, y_3 t_2, \dots , y_n t_n) $. Now transfer an amount $\Delta =y_{n-1} (t_{n-1}-t_{n-2})$ from the first to the 2nd last entry of $V_2$ and continue similarly with the other entries to obtain the required majorization. 
\end{proof}

\begin{lemma}\label{L:3.02} Suppose that $y_1 \geq y_2 \geq \dots \geq y_n > 0$ and $ 0 =t_0 < a = t_1 \leq t_2 \leq \dots \leq t_n =b$. Let $f$ be the function whose value is $y_i$ on $[t_{i-1},t_i)$ and let $P = \int_0^a f = y_1a$ and $Q = \int_a^b f$.  Then $-\int_a^b f(1+ln(xf))dx\geq PlnP -(P+Q)ln(P+Q)$
\end{lemma} 
\begin{proof} Let $g$ be the constant function whose value is $C > 0$ on $[u,v]$ where $0\leq u<v$.  Then $-\int_u^v g(1+ln(xg))dx = (Cu)ln(Cu)-(Cv) ln(Cv)$ with the usual convention that $0 \ln 0 =0$ in the case when $u=0$. Then $-\int_a^b f(1+ln(xf))dx = -\int_0^b f(1+ln(xf))dx +\int_0^a f(1+ln(xf))dx = \sum_{i=1}^n (y_i t_{i-1}  ln( y_i t_{i-1} )- y_i t_i ln(y_i t_i))+PlnP \geq -(P+Q)ln(P+Q) + PlnP $, after applying Lemma~\ref{L:3.01} with $A=P+Q$.

\end{proof} 

\begin{lemma}\label{L:3.03} Let $f$ be a positive monotone decreasing function on $[a,b]$ where $0 <  a < b$ . Let $P = af(a)$ and $Q = \int_a^b f dx$.  Then $-\int_a^b f(1+ln(xf))dx\geq PlnP -(P+Q)ln(P+Q)$.
\end{lemma} 
\begin{proof} By  ~\cite{SH}, there exists a sequence of monotone decreasing step functions $f_k$ that converge uniformly to $f$ with $0 < f(b) \leq f_k(x) \leq f(a)$ for every $k$ and every $x \in [a,b]$. By standard results, the sequence $ f_k(1+ln(xf_k)$ converges uniformly to  $ f(1+ln(xf)$ on $[a,b]$. Let $P_k = af_k(a)$ and  $Q_k  = \int_a^b f_k dx$.  By  Lemma~\ref{L:3.02}, $-\int_a^b f_k(1+ln(xf_k))dx\geq P_klnP_k -(P_k+Q_k)ln(P_k+Q_k)$ and the result now follows by taking the limit as $k \rightarrow \infty$ of both sides of this inequality.
\end{proof}

\begin{lemma}\label{L:3.04} Let $f$ be a positive monotone decreasing function on $(0,b]$ for which $Q = \int_0^b f dx$ is finite.  Then $-\int_0^b f(1+ln(xf))dx\geq  -QlnQ$.
\end{lemma} 
\begin{proof} Suppose $a$ is between $0$ and $b$ and let $P_a = af(a)$, $Q_a = \int_a^b f dx$. Since $Q$ is finite, $lim_{a\rightarrow 0^+} P_a =0$. By Lemma~\ref{L:3.03},  $-\int_a^b f(1+ln(xf))dx\geq P_alnP_a -(P_a+Q_a)ln(P_a+Q_a)$ and the result follows by taking the limit as $a\rightarrow 0^+$ on both sides of this inequality. 
\end{proof}

\begin{lemma}\label{L:3.05} Let $f$ be a positive monotone decreasing function on $(0,\infty)$ for which $Q = \int_0^{\infty} f dx$ is finite.  Then $-\int_0^{\infty} f(1+ln(xf))dx \geq  -QlnQ$.
\end{lemma} 
\begin{proof}  Let $Q_t = \int_0^t f dx$ for $t > 0$,. By Lemma~\ref{L:3.04},  $-\int_0^t f(1+ln(xf))dx \geq - Q_t ln(Q_t)$ and the result follows by taking the limit as $t\rightarrow \infty$ on both sides of this inequality. 
\end{proof}
\begin{theorem}\label{T:3.01} (The Genial Entropy Inequality) For any corner density $f:I \rightarrow [0,\infty)$, $-1-\int_I f(ln(xf)dx\geq 0$.
\end{theorem} 
\begin{proof} Follows from Lemma~\ref{L:3.05} by taking $Q=1$.
\end{proof}
If the density function of a random variable happens to be a corner density, then the Genial Entropy Inequality gives a new lower bound for the differential entropy.
\begin{theorem}\label{T:3.02} Let $X$ be a random variable whose pdf is a corner density $f:I \rightarrow [0,\infty)$ and let $h(X)=-\int_I fln(f)dx$ be the differential entropy of $X$. Then $h(X) \geq 1+ E(\ln X)$.
\end{theorem} 
\begin{proof} Follows from Theorem~\ref{T:3.01}.
\end{proof}

\section{The Slide Derivatives}\label{S:calculus}

With each corner density, we now associate a function called a slide function that describes how the genial entropy changes as the density is deformed to a constant function. In Section~\ref{S:slide}, the slide numbers will be defined as the derivatives of a particular slide functions at $0$. 

\begin{definition}\label{D:4.00} Suppose  $f:(0,b) \rightarrow (0,\infty)$ is monotone decreasing and $A(t) =\int_0^b (f(x))^t dx$.Then  $\sigma_f(t) = G\big(\frac{(f(x))^t}{A(t)} \big) = -1 -\int_0^b \frac{f(x))^t}{A(t)} \ln(x\frac{f(x))^t}{A(t)})dx$ will be called the slide function of $f$ and its domain will be the set of all $t \geq 0$ at which $A$ and $\sigma_f$ both exist.
\end{definition}

By the Genial Entropy Inequality in Theorem~\ref{T:3.01}, we always have $\sigma_f(t)\geq 0$. Note also that $\sigma_f(0)=0$ and if $\int_0^b f=1$ then $\sigma_f(1) = G(f)$. The next theorem says that the function $\sigma_f(t)$ is invariant under changes of scale. It follows from a simple change of variables argument similar to the proof of Proposition~\ref{P:2.02}. 

\begin{theorem}\label{T:4.01} Suppose  $f:(0,b) \rightarrow (0,\infty)$ is monotone decreasing and  $g:(0,\lambda b) \rightarrow (0,\infty)$ is defined by $g(z)=\beta f(\frac{z}{\lambda})$ for some  $\lambda >0$ and any $\beta >0$. Then $\sigma_f=\sigma_g$. 
\end{theorem}

We now show that under mild conditions there must be an $s>0$ for which the interval $[0,s]$ is contained in the domain of $\sigma_f$ and furthermore that $\sigma_f$ must be continuous from the right at $0$.

\begin{lemma}\label{L:4.01} Suppose  $f:(0,b) \rightarrow (0,\infty)$ is monotone decreasing with $\int_0^b (f(x))^s dx < \infty$ for some $s >0$. Then 
$\sigma_f$ is defined on $[0,s/2]$ and continuous from the right at $0$.
\end{lemma} 
\begin{proof} 

We can assume $b=1$ by Proposition~\ref{P:2.02}. Then for $t \in [0,s]$, the function $A(t)$ in Definition~\ref{D:4.00} is finite since $A(t) =\int_0^1 (f(x))^t dx \leq \int_0^1 1+(f(x))^s dx =1+ A(s) < \infty$. To see that $lim_{t\rightarrow 0^+} A(t) =1$, choose $c$ and $d$ with $0 < c < d < 1$ and consider $|A(t)-1|=|\int_0^1 (f(x))^t -1 dx | \leq \int_0^1 |(f(x))^t -1 |dx \leq \int_0^c  (1+f(x)^s) dx + \int_c^d |(f(x))^t -1| dx +\int_d^1 (1+f(x)^s) dx$. In this last sum, the first and third terms are independent of $t$ and can be made as small as desired by choosing $c$ and $d$ appropriately.  Since $|(f(x))^t -1 |$ converges uniformly to $0$ on $[c,d]$ as $t\rightarrow 0^+$, the results follow.

We now show that $\sigma_f(t)$ is finite for $t \in [0,s/2]$ and continuous from the right at $0$ as follows: 
\begin{equation}\notag
\begin{split}
0&\leq \sigma_f(t)\\
&= -1 -\int_0^1 \frac{(f(x))^t}{A(t)} \ln(x\frac{(f(x))^t}{A(t)})dx\\
&=-1 - \frac{1}{A(t)} \int_0^1 (f(x))^t ln(x) dx- \frac{1}{A(t)} \int_0^1 (f(x))^t ln((f(x))^t) dx \\
&+ \frac{1}{A(t)} \int_0^1 (f(x))^t \ln(A(t)) dx \\
&=\frac{1}{A(t)} \big(1-A(t)+A(t)\ln (A(t)) -\int_0^1 ((f(x))^t-1) ln(x) dx \\
&-\int_0^1 (f(x))^t ln((f(x))^t) dx \big)\\
&\leq \frac{1}{A(t)} \big(|1-A(t)+A(t)\ln (A(t)| +\int_0^1 |(f(x))^t-1|| \ln(x)| dx \\
&+\int_0^1 |(f(x))^t \ln((f(x))^t)| dx \big)
\end{split}
\end{equation}

It remains to show that each of the three terms in this last sum is finite for $t \in [0,s/2]$ and goes to $0$ as $t\rightarrow 0^+$. For the first term, clearly \\  $\lim_{t\rightarrow 0^+} \big(1-A(t)+A(t)\ln (A(t) \big)=0$.

For the second term, $\int_0^1 |(f(x))^t-1|| \ln(x)|dx$ converges since the integrand is the product of square integrable functions. To see that $\lim_{t\rightarrow 0^+} \int_0^1 |(f(x))^t-1|| \ln(x)| dx =0$, choose $c$ and $d$ with $0 < c < d < 1$ and consider the inequality
$\int_0^1 |(f(x))^t-1|| \ln(x)| dx \leq \int_0^c (1+f(x)^{s/2})| \ln(x)| dx+\int_c^d |(f(x))^t-1|| \ln(x)| dx+\int_d^1 (1+f(x)^{s/2})| \ln(x)| dx$. Now follow the argument given above for $|A(t)-1|$.

For the third term, use the inequality $z-1 \leq z\ln z \leq z(z-1)$ for $z \geq 0$ to get  $ \int_0^1 |(f(x))^t \ln((f(x))^t)| dx \leq \int_0^1 |(f(x))^t-1| \max(1,(f(x))^t) dx $ which converges since the last integrand is the product of square integrable functions. Now show $lim_{t\rightarrow 0^+} \int_0^1 |(f(x))^t-1| \max(1 , (f(x))^t) dx =0$ as before. 

\end{proof}

The information we wish to extract from a corner density is captured by the derivatives of its slide function at $0$ which we now describe.

\begin{definition} \label{D:4.01} Suppose  $f:(0,b) \rightarrow (0,\infty)$ is monotone decreasing with $\int_0^b (f(x))^s dx < \infty$ for some  $s>0$. Then the $n$'th slide derivative of $f$ is defined by $\psi_n(f)=\frac{d^n \sigma_{f}}{dt^n}(0)$ where all derivatives are taken from the right. If all of these derivatives exist, then the slide series of $f$ is defined to be $\sum_{i=1}^{\infty} \frac{\psi_n(f)}{n!}t^n$.

\end{definition}

Here are some elementary properties of $\psi_n(f)$.

\begin{theorem}\label{T:4.02}  Suppose  $f:(0,b) \rightarrow (0,\infty)$ is monotone decreasing with $\int_0^b (f(x))^s dx < \infty$ for some  $s>0$.
\begin{enumerate}

\item If $\psi_1(f)$ exists then $\psi_1(f) \geq 0$.

\item If $\psi_n(f)$ exists then so does $\psi_n(f^r) $ for $r>0$ and $\psi_n(f^r) =r^n \psi_n(f)  $.  

\item If $f$ is a constant function, then $\psi_n(f) =0$ for all n.
\end{enumerate}
\end{theorem} 
\begin{proof} 

(1) The slide function $\sigma_f$ is non-negative on its domain and $\sigma_f(0)=0$ so the first derivative must be nonnegative.

(2) For $t$ sufficiently small and nonnegative, we have $\sigma_{f^r}(t)=\sigma_f(rt)$ and the result follows from the chain rule.  

(3) If $f$ is a constant function, then $\sigma_f(t)=0$ for all $t \geq 0$.

\end{proof}

Here is an example of a slide derivative calculation that we will connect with the uniform distribution on $[0,1]$ in Section~\ref{S:slide}.

\begin{theorem}\label{T:4.03} Let $f(x)=-\ln(x) $ for $x \in (0,1)$. Then for $t > 0$,  $\sigma_f(t)=-1+t-t\Psi(t)+\log(\Gamma(1+t))$ and the slide derivatives of $f$ are given by $\psi_1(f)=1$ and by  $\psi_n(f)=(-1)^{n+1}(n-1)!(n-1)\zeta(n)$ for $n>1$ .
\end{theorem} 
\begin{proof} 

Let $g_t(x)=\frac{(f(x))^t}{A(t)}=\frac{(-\ln(x))^t }{\Gamma(1+t)}$ so $\int_0^1 g_t(x)dx =1$. Then
\begin{equation}\label{xx}
\begin{split}
\sigma_f(t)&= G(g_t(x))\\
&= -1 -\int_0^1g_t\ln(xg_t)dx\\
& = -1 -\int_0^1g_t\ln(x)dx -\int_0^1g_t\ln(g_t)dx\\
 &= -1 -(-1-t)-(1-\ln(\Gamma(1+t))+t\Psi(t))\\
  &= -1 +t+\ln(\Gamma(1+t))-t\Psi(t)
\end{split}
\end{equation}

The result now follows by differentiating $-1 +t+\ln(\Gamma(1+t))-t\Psi(t)$ from the right at $t=0$.
\end{proof}

\begin{corollary}\label{C:4.01} Let $f_r(x)=(-\ln(x))^r$ for $x \in (0,1)$ and $r>0$. Then $\psi_1(f_r)=r$ and $\psi_n(f_r)=(-1)^{n+1}(n-1)!(n-1)\zeta(n)r^n$ for $n>1$. The slide series for $f_r$ is given by $rt+  \sum_{n=2}^{\infty}\frac{(-1)^{n+1} (n-1)!(n-1)\zeta(n) r^n t^n}{n!}$.
\end{corollary} 

\begin{proof} 
Immediate from Theorem~\ref{T:4.02} and Theorem~\ref{T:4.03}.

\end{proof}

\begin{example} \label{E:4.01} $\psi_1(f)=\infty$ for the function $f(x)=exp(-1/(1-x)^2)$  on $[0,1)$. 
\end{example}

In order to apply the slide derivatives to samples in a metric space, we first find the distance $d_i$ from the $i$'th point to its nearest neighbour.  These distances can be used to construct the function $f_D$ in Definition~\ref{D:2.03} and we can then compute the first slide derivative $\psi_1(f_D)$ using the explicit formula provided in the next theorem. 

\begin{theorem}\label{T:4.04} Suppose  $ D=\{d_1 ,d_2, \dots, d_n \} $ where we assume $d_1 \geq d_2 \geq \dots \geq d_n > 0$ and let $f_D$ be the function on $[0,1)$ whose value on the interval $[\frac{i-1}{n},\frac{ i}{n})$ is $d_i$ as in Definition~\ref{D:2.03}.  Then the first slide derivative of $f_D$ is given by $\psi_1(f_D)= \frac{1}{n}\sum_{i=2}^{n-1}i\ln(i) \ln(\frac{d_{i+1}}{d_i})+\frac{\ln(n)}{n}\sum_{i=1}^{n-1} \ln(\frac{d_i}{d_n})$.
\end{theorem} 
\begin{proof} 
By Definition~\ref{D:4.01}, we have to calculate the right hand derivative of the slide function $\sigma_{f_D}(t)$ at $t=0$ so we first find an expression for $\sigma_{f_D}(t)$.  Let $g_t(x)=\frac{(f_D(x))^t}{A(t)}$ where $A(t) = \int_0^1 (f_D(x))^t dx$ so $g_t(x)$ takes the value $a_i(t)$ on $[\frac{i-1}{n},\frac{ i}{n})$ where $a_i(t)=\frac{nd_i^t}{\sum_{i=1}^n d_i^t}$.

\begin{equation}\notag
\begin{split}
\sigma_f(t)&=G(g_t(x))\\
&=-1-\int_0^1 g_t(x)\ln(xg_t(x))dx\\
&=-1-\sum_{i=1}^n \int_{\frac{i-1}{n}}^{\frac{i}{n}} g_t(x)\ln(xg_t(x))dx\\
&=-1-\big(\big(\frac{a_1}{n}\big)\ln \big(\frac{a_1}{n}\big)-\big(\frac{a_1}{n}\big)+\sum_{i=2}^n \big(\big(\frac{ia_i}{n}\big)\ln\big(\frac{ia_i}{n}\big)-\big(\frac{(i-1)a_i}{n}\big)\ln\big(\frac{(i-1)a_i}{n}\big)-\frac{a_i}{n}\big)\big)\\
&=-\big(\frac{a_1}{n}\big)\ln \big(\frac{a_1}{n}\big)-\sum_{i=2}^n \big(\frac{ia_i}{n}\big)\ln\big(\frac{ia_i}{n}\big)+\sum_{i=2}^n \big(\frac{(i-1)a_i}{n}\big)\ln\big(\frac{(i-1)a_i}{n}\big)
\end{split}
\end{equation}
To find the derivative of $\sigma_f(t)$, we now use the facts that $\frac{d a_k}{dt} = \frac{n-1}{n}\ln(d_k)-\frac{1}{n}\sum_{i\neq k}^n \ln(d_i)$ and $\frac{d(v \ln v)}{dt}=(1+\ln v)\frac{dv}{dt}$. 
\begin{equation}\notag
\begin{split}\frac{d \sigma_f}{dt}(0)&= \frac{d}{dt} \big( \big(\frac{a_1}{n}\big)\ln \big(\frac{a_1}{n}\big)-\sum_{i=2}^n \big(\frac{ia_i}{n}\big)\ln\big(\frac{ia_i}{n}\big)+\sum_{i=2}^n \big(\frac{(i-1)a_i}{n}\big)\ln\big(\frac{(i-1)a_i}{n}\big) \big)_{t=0}\\
&=-\big(1+\ln \big(\frac{a_1(0)}{n}\big)\big) \big( \frac{1}{n}\big)\big( \big(\frac{n-1}{n}\big)\ln(d_1)-\frac{1}{n}\sum_{j\neq 1}^n \ln(d_j) \big)\\
&-\sum _{i=2}^{n}\big(\big(1+\ln \big(\frac{ia_i(0)}{n}\big)\big) \big( \frac{i}{n}\big)\big( \big(\frac{n-1}{n}\big)\ln(d_i)-\frac{1}{n}\sum_{j\neq i}^n \ln(d_j) \big)\big)\\
&+ \sum _{i=2}^{n}\big(\big(1+\ln \big(\frac{(i-1)a_i(0)}{n}\big)\big) \big( \frac{i-1}{n}\big)\big( \big(\frac{n-1}{n}\big)\ln(d_i)-\frac{1}{n}\sum_{j\neq i}^n \ln(d_j) \big)\big)
\end{split}
\end{equation}
At $t=0$, each $a_i$ is equal to $1$ so this derivative becomes
\begin{equation}\notag
\begin{split}
\frac{d \sigma_f}{dt}(0)&= -\big(1+\ln \big(\frac{1}{n}\big)\big) \big( \frac{1}{n}\big)\big( \big(\frac{n-1}{n}\big)\ln(d_1)-\frac{1}{n}\sum_{j\neq 1}^n \ln(d_j) \big)\\
&-\sum _{i=2}^{n}\big(\big(1+\ln \big(\frac{i}{n}\big)\big) \big( \frac{i}{n}\big)\big( \big(\frac{n-1}{n}\big)\ln(d_i)-\frac{1}{n}\sum_{j\neq i}^n \ln(d_j) \big)\big)\\
&+\sum _{i=2}^{n}\big(\big(1+\ln \big(\frac{i-1}{n}\big)\big) \big( \frac{i-1}{n}\big)\big( \big(\frac{n-1}{n}\big)\ln(d_i)-\frac{1}{n}\sum_{j\neq i}^n \ln(d_j) \big)\big)\\
&=-\big(1+\ln \big(\frac{1}{n}\big)\big) \big( \frac{1}{n}\big)\big( \big(\frac{n-1}{n}\big)\ln(d_1)-\frac{1}{n}\sum_{j\neq 1}^n \ln(d_j) \big)\\
&+\sum _{i=2}^{n}\big(\big(-1+(i-1)\ln(i-1)-i\ln(i) +\ln(n)\big) \big( \frac{1}{n}\big)\big( \big(\frac{n-1}{n}\big)\ln(d_i)-\frac{1}{n}\sum_{j\neq i}^n \ln(d_j) \big)\big)\\
&=P_1+P_2+P_3
\end{split}
\end{equation}
The $P_i$ terms are defined and calculated as follows.
The $P_1$ term consists of the parts of this last expression that involve the isolated $1$s.
\begin{equation}\notag
\begin{split}
P_1&= -\big(1\big) \big( \frac{1}{n}\big)\big( \big(\frac{n-1}{n}\big)\ln(d_1)-\frac{1}{n}\sum_{j\neq 1}^n \ln(d_j) \big)\\
& +\sum _{i=2}^{n}\big(\big(-1\big) \big( \frac{1}{n}\big)\big( \big(\frac{n-1}{n}\big)\ln(d_i)-\frac{1}{n}\sum_{j\neq i}^n \ln(d_j)
 \big)\big)\\
 &=0
\end{split}
\end{equation}
$P_2$ is the sum of all of the terms containing $\ln n$.
\begin{equation}\notag
\begin{split}
P_2&= \big( \frac{\ln n}{n^2}\big)\big( \big(n-1\big)\ln(d_1)-\sum_{j\neq 1}^n \ln(d_j) \big)\\
& +\big( \frac{\ln n}{n^2}\big)\sum _{i=2}^{n}\big( \big( \big(n-1\big)\ln(d_i)-\sum_{j\neq i}^n \ln(d_j)
 \big)\big)\\
 &+\big( \frac{-n\ln n}{n^2}\big)\big( \big(n-1\big)\ln(d_n)-\sum_{j\neq n}^n \ln(d_j) \big)\\
 &= \big( \frac{-n\ln n}{n^2}\big)\big( \big(n-1\big)\ln(d_n)-\sum_{j\neq n}^n \ln(d_j) \big)\\
&=\frac{\ln(n)}{n}\sum_{i=1}^{n-1} \ln(\frac{d_i}{d_n})
\end{split}
\end{equation}
$P_3$ is what remains after $P_1$ and  $P_2$ are subtracted from $\frac{d \sigma_f}{dt}(0)$.
\begin{equation}\notag
\begin{split}
P_3&= \sum _{i=2}^{n}\big(\big((i-1)\ln(i-1)-i\ln(i)\big) \big( \frac{1}{n}\big)\big( \big(\frac{n-1}{n}\big)\ln(d_i)-\frac{1}{n}\sum_{j\neq i}^n \ln(d_j) \big)\big)\\
&+\big( \frac{n\ln n}{n^2}\big)\big( \big(n-1\big)\ln(d_n)-\sum_{j\neq n}^n \ln(d_j) \big)\\
&=\frac{1}{n^2} \sum_{i=2}^{n-1} i \ln i \big( \big(-(n-1) \ln(d_i)+\sum_{j\neq i}^n \ln (d_j) \big) +\big((n-1) \ln(d_{i+1})-\sum_{j\neq {i+1}}^n \ln (d_j)   \big)  \big)\\
&=\frac{1}{n^2} \sum_{i=2}^{n-1} i \ln i \big( -n\ln(d_i)+n \ln(d_{i+1}) \big)\\
&= \frac{1}{n}\sum_{i=2}^{n-1}i\ln(i) \ln(\frac{d_{i+1}}{d_i})
\end{split}
\end{equation}
\end{proof}

The following formula for the second slide derivative $\psi_2(f)$ is motivated by calculations for small values of $n$. In the next section, we will see that the results from simulations based on this formula agree with what we would expect from theoretical considerations. 

\begin{conjecture}\label{CJ:4.01} Suppose  $ d_1 \geq d_2 \geq \dots \geq d_n > 0$ and let $f_D$ be the function on $[0,1)$ whose value on the interval $[\frac{i-1}{n},\frac{ i}{n})$ is $d_i$. Let $S_1=\sum_{i=1}^n \log(d_i)$, $S_2=\sum_{i=1}^n \log(d_i)^2$ and $S_3=\sum_{i=1}^{n-1} log(d_i/d_n)^2$. Then the second slide derivative of $f_D$ is given by 
\begin{multline*}
\psi_2(f_D)= - \big( \sum_{i=1}^{n-1} \big( i \log(i) \log(d_{i+1}/d_i)(2S_1-n \log(  d_i d_{i+1}))\big)\\
+\log(n) (2(S_1-n \log(d_n))^2-nS_3) +nS_2 - S_1^2 \big)/n^2.
\end{multline*}
\end{conjecture}

The next section define the slide numbers and illustrates their application to some standard point processes. 

\section{The Slide Numbers}\label{S:slide}
Sample data often consists of a set of distinct points in a metric space.  With the help of Definition~\ref{D:4.01}, we can now associate with each of these samples an infinite family of new statistics called the slide numbers. 

\begin{definition}\label{D:5.01} Let M be a metric space and let $U= \{u_1, u_2, \dots, u_k \}$ be a set of k distinct points in $M$. For each $i=1, \dots, k$, let $d_i$ be the distance from $u_i$ to its nearest neighbour in $U$.  Define a sequence $D$ by ordering the $d_i$ in descending order as $ d_{[1]} \geq d_{[2]} \geq \dots \geq d_{[k]} >0$. As in Definition~\ref{D:2.03}, let $f_{D}$ be the function on $[0,1)$ whose value on the interval $[\frac{i-1}{k},\frac{ i}{k})$ is $d_{[i]}$. Define the $n$'th slide number of $U$ by $\rho_n(U) = \psi_n(f_{D})$ and define the slide series of $U$ to be $\sum_{i=1}^{\infty} \frac{\rho_n(U)}{n!}t^n$.
\end{definition}

Values of $\rho_1(U)$ for various random variables are shown in Table~\ref{T:5.01} which shows the connection between $1/\rho_1(U)$ and the Hausdorff dimension of $[0,1]^m$, the Cantor set and the Sierpinski triangle. For certain point processes P, the statistics $\rho_n(U)$ appear to converge as the sample size $U$ gets large.  For example, for any normal random variable $Z$, the quantity $\rho_1(U)$ appears to converge to $4/\pi$ in which case it makes sense to write $\rho_1(Z)=4/\pi$. More generally we define $\rho_n(P)$ for an arbitrary point process $P$ as follows.

\begin{definition}\label{D:5.03} Let $M$ be a metric space and let $U = \{u_1, u_2, \dots \}$ be a sequence of distinct points in $M$ generated by some point process $P$ and let $ U_k = \{u_1, u_2, \dots, u_k \}$.  If $\rho_n(U_k)$ converges in probability as $k \rightarrow \infty$, then $\rho_n(P)$ is defined to be the value of this limit. If all of the limits $\rho_n(P)$ exist, then we define the slide series of the process $P$ to be $\sum_{i=1}^{\infty} \frac{\rho_n(P)}{n!}t^n$. In the case where $U$ is a sample of a random variable $X$, we will use the notation $\rho_n(X)$ instead of $\rho_n(P)$.
\end{definition}

The following proposition follows immediately from the definitions and shows that adjusting the mean or standard deviation of a random variable $X$ has no effect on $\rho_n(X)$. 

\begin{proposition}\label{P:5.01}If $X$ is a random variable, then $\rho_n(aX+b)=\rho_n(X)$ for all real number $a$ and $b$ with $a \neq 0$.   
\end{proposition}

Some evidence for the convergence of the slide statistics $\rho_n(U)$ is given in  Table~\ref{T:5.01} and Table~\ref{T:5.02}. In these tables, the points in the Cantor set were generated using $\sum_{i=1}^{40} \frac{a_i}{3^i}$  where the $a_i$ are either $0$ or $2$ with probability $1/2$. Points in the Sierpinski triangle were generated using the Chaos Game ~\cite{BM}. The generation of all random numbers used in these simulations was based on the Mersenne Twister.

\begin{table}[!htb]{Simulated Values of $\rho_1(U)$ for Various Random Variables}
\begin{center}
\begin{tabular}{|c|c|c|c|c|}
 \hline
\textit{Density} & $\mu_1$ &  $\sigma_1$ & $1/\mu_1$  & $\frac{1}{\rho_1}$ \textit{conjectured}    \\
\hline
uniform on $[a,b]$ & $1.0003$ & $0.0111$  & $0.9997$  & $1$   \\
\hline
{normal} & $1.2664$ & $0.129 $  & $0.7896$ & $\pi/4 \simeq 0.785 $  \\
\hline
exponential  & $1.4590$ & $0.0141$ & $0.6854$  & ?  \\
\hline
$1/(2\sqrt{x})$ on $[0,1]$  & $1.2817$ & $0.0132 $  & $0.7802$ & ?  \\
\hline
uniform on $[0,1]^2$ & $0.5023$ & $0.0056$  & $1.9908$ & $2$   \\
\hline
uniform on $[0,1]^3$ & $0.3416$  & $0.0037$  & $2.9274$ & $3$   \\
\hline
uniform on $[0,1]^4$ & $0.2642$ & $0.0029$  & $3.7850$ & $4$   \\
\hline
bivariate normal & $0.7264$ & $0.0073 $  & $1.3766$ & ? \\
\hline
Cantor  & $1.6014$ & $0.0170 $  & $0.6244$ & $\ln(2)/\ln(3)  \approx 0.631$   \\
\hline
Sierpinski  & $0.6344$ & $0.0067 $  & $1.57624$ & $\ln(3)/\ln(2)  \approx 1.5849$   \\
\hline

\end{tabular}
\caption{ For each density, $1000$ samples of size $10000$ were generated and the value of $\rho_1(U)$ was computed for each sample $U$. The mean $\mu_1$ and standard deviation $\sigma_1$ of these $1000$ values for $\rho_1(U)$ are shown. The value given for $1/\rho_1$ is the conjectured limiting value of $1/\mu_1$ as the sample size approaches infinity.}\label{T:5.01}
\end{center}
\end{table}

\begin{table}[!htb]{Simulated Values of $\rho_2(U)$ for Various Random Variables}
\begin{center}

\begin{tabular}{|c|c|c|c|c|}
 \hline
\textit{Density} & $\mu_2$ &  $\sigma_2$  &  $\rho_2$ \textit{conjectured}   \\
\hline
uniform on $[a,b]$ & $-1.6461$ & $.0732$   & $-\pi^2/6 \approx -1.6449$   \\
\hline
normal & $-1.0273$ & $0.0860 $   & $-1$  \\
\hline
exponential & $-0.7333$ & $0.0920 $   & ?  \\
\hline
$1/(2\sqrt{x})$ on $[0,1]$  & $-2.5792$ & $0.1085$  & ? \\
\hline
uniform on $[0,1]^2$ & $-0.4096$ & $0.0186$   & $-(\pi^2/6)(1/2)^2 \approx -.4112$   \\
\hline
uniform on $[0,1]^3$ & $-0.1825$  & $0.0083$   & $-(\pi^2/6)(1/3)^2 \approx -0.1827$   \\
\hline
uniform on $[0,1]^4$ & $-0.1038$ & $0.0049$   & $-(\pi^2/6)(1/4)^2 \approx 0.1028$   \\
\hline
bivariate normal & $-0.2004$ & $0.0233 $   & ?  \\
\hline
Cantor & $-4.1464$ & $0.1933 $   &  $ \frac{(-1)^{2+1}(2-1)!(2-1)\zeta(2)}{(ln(2)/ln(3))^2} \approx -4.132$   \\ 
\hline
Sierpinski & $-0.6549$ & $0.0295 $   &  $ \frac{(-1)^{2+1}(2-1)!(2-1)\zeta(2)}{(ln(3)/ln(2))^2} \approx -0.655$   \\ 
\hline

\end{tabular}
\caption{ For each density, $1000$ samples of size $10000$ were generated and the value of $\rho_2(U)$ was computed for each sample $U$ using   Conjecture~\ref{CJ:4.01}. The mean $\mu_2$ and standard deviation $\sigma_2$ of these $1000$ values for $\rho_2(U)$ are shown. The value given for $\rho_2$ is the conjectured limiting value of $\mu_2$ as the sample size approaches infinity. }\label{T:5.02}
\end{center}
\end{table}

Consider the following thin outline for a possible argument to explain the empirical results obtained for $[0,1]^m$ in Table~\ref{T:5.01} and Table~\ref{T:5.02}.  Suppose $U = \{u_1, u_2, \dots, u_k \}$ is a very large sample of points $(x_1,x_2,\dots,x_m)$ in $[0,1]^m$ where the $x_i$ are chosen independently and uniformly at random from $[0,1]$,  and let $u$ be a particular point in $U$. By \cite{SKM}, the probability that a point is within $r$ of $u$ is approximately $1-e^{-\alpha r^m}$ for some $\alpha$. If the sample $U$ is large enough, then the set of nearest neighbour distances will be sufficiently independent \cite{K,S} that their empirical cumulative distribution will also be approximately equal to $1-e^{-\alpha r^m}$. If $D$ is the ordered sequence of nearest neighbour distances, then $L_{D^*}(r)$ in Definition~\ref{D:2.03} will be approximately $1-e^{-\beta r^m}$ for some $\beta$ and $1-L_{D^*}(r)$ will be approximately $e^{-\beta r^m}$ with inverse $f_{D^*}(x) = \frac{(-\log(x))^\frac{1}{m}}{\Gamma(1+\frac{1}{m})}$. Now $f_{D^*}(x)$ has the same genial entropy as $1-L_{D^*}(r)$ by Theorem~\ref{T:2.03} so Corollary~\ref{C:4.01} now suggests the following conjecture which is supported by the empirical results shown in Table~\ref{T:5.01} and Table~\ref{T:5.02}.

\begin{conjecture}\label{CJ:5.01} Let $U = (u_1, u_2, \dots )$ be a sequence of points $(x_1,x_2,\dots,x_m)$ in $[0,1]^m$, where the $x_i$ are chosen independently and uniformly at random from $[0,1]$ and let $ U_k = \{u_1, u_2, \dots, u_k \}$. Then as $k \rightarrow \infty$,  $\rho_1(U_k)$ converges in probability to $1/m$ and  $\rho_n(U_k)$ converges in probability to $(-1)^{n+1}(n-1)!(n-1)\zeta(n)/m^n$ for $n >1$.
\end{conjecture} 

There appear to be cases other than $[0,1]^m$ for which the dimension equals $m$ and $\rho_n$ converges to $(-1)^{n+1}(n-1)!(n-1)\zeta(n)/m^n$.  For example, if we take $m=\log(2)/\log(3)$ then $(-1)^{2+1}(2-1)!(2-1)\zeta(2)/m^2 \approx -4.132$ which is close to the value shown in Table~\ref{T:5.02} for the  Cantor set. A similar result holds for the Sierpinski triangle and prompts us to make the following definition. 

\begin{definition}\label{D:5.04} In the context of Definition~\ref{D:5.03}, we say that a point process $P$ is tangible provided there is a number $d$ with $\rho_1(P)=1/d$  and $\rho_n(P)=(-1)^{n+1}(n-1)!(n-1)\zeta(n)/d^n$ for $n>1$. The number $d$ will be called the slide dimension of the process.  If there is no such number, the process will be called intangible. 
\end{definition}

The values for $\rho_1$ and $\rho_2$ given in Table~\ref{T:5.01} and Table~\ref{T:5.02} suggest that the normal distribution does not satisfy the conditions for tangibility in  Definition~\ref{D:5.04} so cannot be assigned a slide dimension. In the case of a tangible process, we can recover the dimension by rearranging the second derivative to obtain  $d = \frac{\pi}{\sqrt{-6\rho_2}}$.  This relationship is particularly useful for spatial statistics because it provides considerably better estimates for the dimension of $[0,1]^m$ than the values for $1/\mu_1$ shown in Table~\ref{T:5.01}. For a tangible process $P$, we would like to know in general if the statistics $\sqrt[n]{\frac{(-1)^{n+1}(n-1)!(n-1)\zeta(n)}{\rho_n}}$ converge more quickly to the dimension for larger values of $n$.

For the subset $U$ of $R$ generated by $20000$ iterations of $x_{i+1}=x_i+cos(i)$ with $x_0=0$, the value of $\rho_1(U)$ is approximately $0.53$. But values of $\rho_1(U)$ less than $1$ cannot occur for continuous real-valued random variables according to the next conjecture which is supported by the results in Table~\ref{T:5.01}. 

\begin{conjecture}\label{CJ:5.03} If $X$ is a continuous real-valued random variable for which $\rho_1(X)$ exists, then $\rho_1(X) \geq 1$.
\end{conjecture} 
If $U$ consists of the first $20,000,000$ primes, then the value of $\rho_1(U)$ is approximately $0.77235$ which is interesting in view of Conjecture~\ref{CJ:5.03}. 

The values of $\rho_2(X)$ are all negative in Table~\ref{T:5.02} but $\rho_2(X)$ appears to be positive when $X$ has a Cauchy distribution which raises the following question.

\begin{question}\label{Q:5.04} What conditions on a random variable $X$ determine the sign of $\rho_n(X)$ when $n>1$?
\end{question}

\section{Applications of the slide statistics to financial returns.}\label{S:returns}
If the daily closing prices of an index or stock are $X_1, X_2, \dots$ and the daily returns are given by $U_i =\ln(X_{i+1}/X_i)$, then a central question of financial mathematics is the problem of describing the distribution of the returns $U_i$. We now show that the slide statistics impose strict constraints that allow us to quickly reject many possible candidates for this return distribution.  Our approach will be to fix a sample size $r$ and use the standard metric on $R^n$ to calculate $\rho_1$ and $\rho_2$ for the subset of $R^n$ given by $T_n =\{(U_1, U_2, ... , U_n), (U_2, U_3, ... , U_{n+1}), \dots,(U_r, U_3, ... , U_{n+r-1})\}$. Note that by the same reasoning used in Proposition~\ref{P:5.01}, the values of  $\rho_i(T_n)$  are not changed if we replace the $U_i$ by $aU_i +b$ with $a \neq 0$. In other words, $\rho_i(T_n)$ is detecting information about the returns that has nothing to do with their mean and standard deviation.

Figure~\ref{fig:Fig1} shows that the values of $\rho_1(T_n)$ obtained for the S$\&$P $500$ are very different than the values obtained for either the Normal or Laplace distributions. The lowest of the $11$ curves shown in Figure~\ref{fig:Fig1} is the plot of $\rho_1(T_n)$ against $n$ for the S$\&$P $500$ for the $5000$ trading days ending December 31, 2014.  The middle cluster of five curves are the corresponding plots obtained for five different simulated sequences of returns having the Laplace distribution. The upper cluster of five curves was obtained for the Normal distribution and we see that these curves are a bad fit to the $\rho_1$ curve for the S$\&$P $500$. Any suggested model for the returns of the S$\&$P $500$ must be able to generate an approximation of this $\rho_1$ curve which places a tight constraint on potential stochastic processes. In Figure~\ref{fig:Fig2}, the values of $\rho_2(T_n)$ shown on the upper graph were obtained for the S$\&$P $500$  for the $5000$ trading days ending December 31, 2014. For this particular time period, $\rho_2(T_n)$ was positive but the values obtained for samples from Normal distributions are typically negative as shown by the $5$ graphs below the horizontal axis in Figure~\ref{fig:Fig2}. 

\begin{figure}[!ht]
\centering
\includegraphics[width=4in]{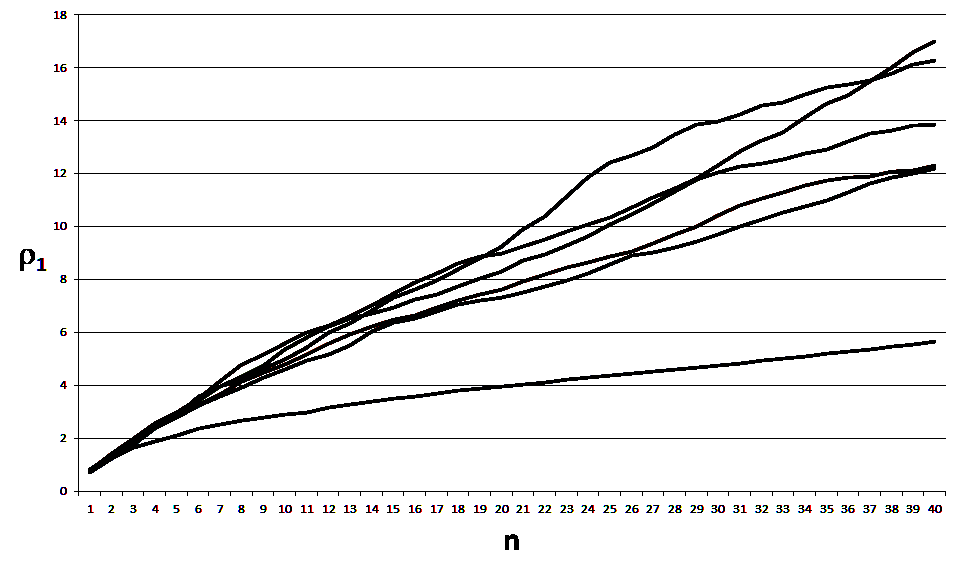}
\caption{The lower graph is the plot of  $\rho_1(T_n)$ against $n$ for the $500$ monthly returns of the S$\&$P $500$ up to the end of $2014$. The upper cluster of five graphs are the corresponding plots obtained from five different simulated sequences of normally distributed returns. }
\label{fig:Fig2}
\end{figure}

The failure of the daily returns to be normal might be due to dependencies in the consecutive daily returns and it is possible that we might obtain a better fit using the monthly returns instead. The bottom curve in Figure~\ref{fig:Fig2} shows $\rho_1(T_n)$ plotted against $n$ for the $500$ monthly returns of the S$\&$P $500$ up to the end of $2014$.  The group of five curves at the top of Figure~\ref{fig:Fig2} are the corresponding plots obtained for five different simulated sequences of returns having a Normal distribution. The fit is better than what we obtained for the daily returns but once again the normal $\rho_1$ curves are not a match for the $\rho_1$ curve for the S$\&$P $500$ which is consistent with the non-normality of such returns found in \cite{BP}.  As another point of view on the non-normality of the monthly returns,  Figure~\ref{fig:Fig3} shows the points $(\rho_2(T_3),\rho_1(T_3))$ plotted for $100,000$ samples from a normal distribution as well as the point obtained for the $500$ monthly returns of the S$\&$P $500$ up to the end of $2014$. The possibility of describing these monthly returns as a mixture of two Gaussians is considered in \cite{BP} but we find $\rho_2(T_3) <0$ for their models but we sometimes find $\rho_2(T_3) >0$ for the S$\&$P $500$.

\begin{figure}[!ht]
\centering
\includegraphics[width=4in]{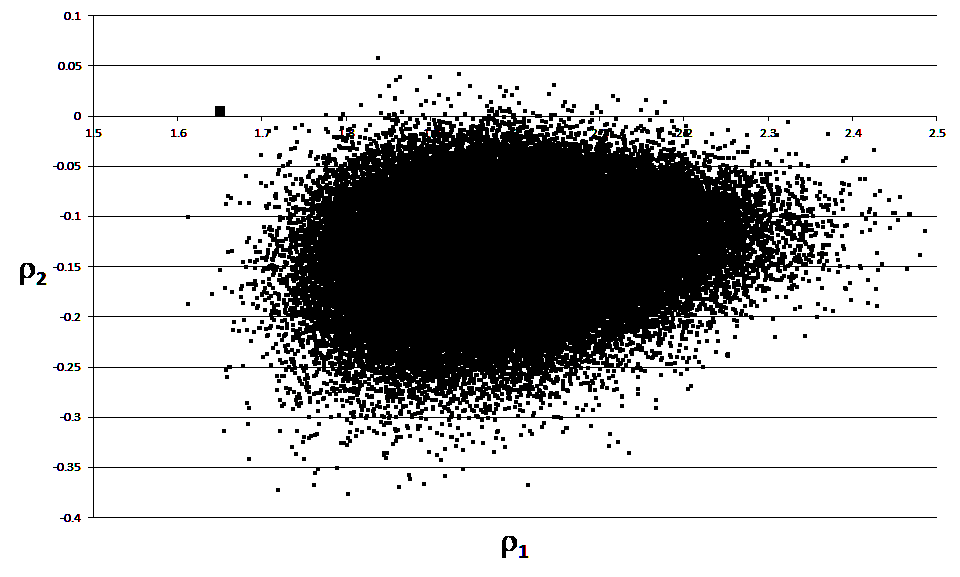}
\caption{The small square at the upper left plots the point $(\rho_2(T_3),\rho_1(T_3))$ obtained for the $500$ monthly returns of the S$\&$P $500$ ending with December of $2014$. The dots are the points $(\rho_2(T_3),\rho_1(T_3))$ obtained from $100,000$ different simulated sequences of normally distributed returns.  }
\label{fig:Fig3}
\end{figure}

\begin{figure}[!ht]
\centering
\includegraphics[width=4in]{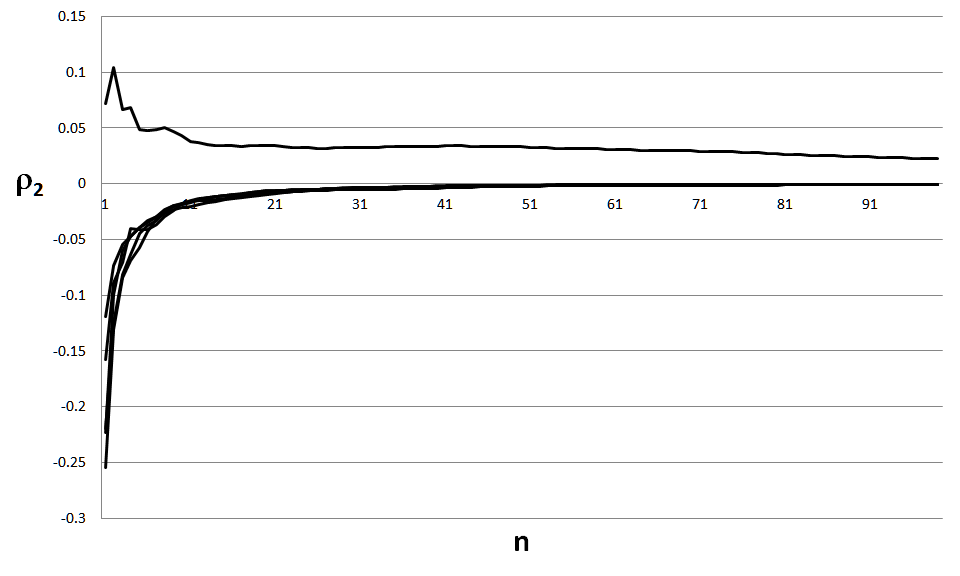}
\caption{The upper graph is the plot of  $\rho_2(T_n)$ against $n$ for the S$\&$P $500$ for the $5000$ trading days ending December 31, 2014.  The lower cluster of five graphs are the corresponding plots obtained from five different simulated sequences of normally distributed returns. }
\label{fig:Fig4}
\end{figure}

\begin{table}[!ht]{Values of $\rho_1$ and $\rho_2$ for Various Stocks and Indexes}
\begin{center}
\begin{tabular}{|c|c|c|}
 \hline
\textit{Name} & $\rho_1(T_3)$ &  $\rho_2(T_3)$     \\
\hline
Exxon & $1.44659$ & $-0.12179$     \\
\hline

\hline
FTSE & $1.34232$ & $-0.08638$     \\
\hline

\hline
IBM & $1.32328$ & $0.011604$     \\
\hline

\hline
3M & $1.38742$ & $-0.07126$     \\
\hline

\hline
NASDAQ & $1.24174$ & $0.03204$     \\
\hline

\hline
Russell $2000$ & $1.33124$ & $-0.12882$     \\
\hline

\hline
S$\&$P $500$ & $1.41405$ & $-0.02228$     \\
\hline

\end{tabular}
\caption{ Each of these values was calculated using the 5000 trading days ending with June 30, 2008. }\label{T:5.03}
\end{center}
\end{table}

The stable distributions are often considered ~\cite{SB,BKS,HP} as possible models for financial returns.  We now consider the family of stable distributions $S(\alpha,\beta,\gamma,\delta)$ described in  ~\cite{W} where $\alpha \in (0,2]$ is the stable parameter, $\beta \in [-1,1]$ is the skewness parameter, $\gamma >0 $ is the scale parameter and $\delta  \in (-\infty,\infty)$ is the location parameter. By Proposition~\ref{P:5.01},  the values of $\rho_i$ are not affected by changes in $\gamma $  and $\delta$ so we will simply work with the family $S(\alpha,\beta,1,0)$ which we write as $S(\alpha,\beta)$. Figure~\ref{fig:Fig5} shows the points $(\rho_2(T_3),\rho_1(T_3))$ from Table~\ref{T:5.03} plotted together with $100,000$ samples from stable distributions with parameters $\alpha$ and $\beta$ chosen uniformly at random from the intervals $(1,2)$ and $(0,1)$ respectively. We see that the values in Table~\ref{T:5.03} lie outside the region corresponding to the stable distributions.  Figure~\ref{fig:Fig6} show an expanded view of Figure~\ref{fig:Fig5} showing that the slide statistics for the financial returns in Table~\ref{T:5.03}  are well outside the region corresponding to samples from stable distributions. This failure of the stable distributions to fit financial returns is consistent with the findings of \cite{LLW,BKS}.

\begin{figure}[!ht]
\centering
\includegraphics[width=4in]{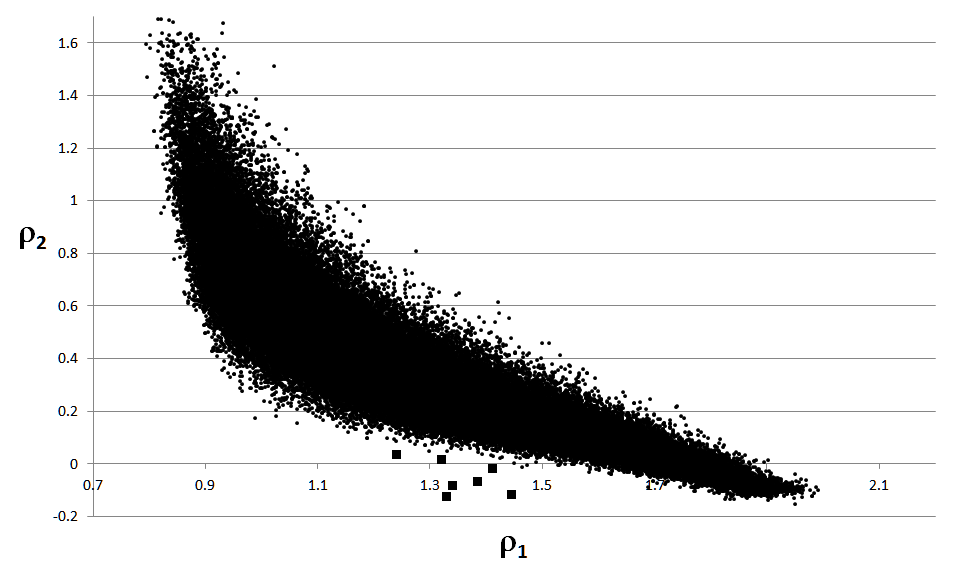}
\caption{The dots plotted here are the points $(\rho_1(T_3),\rho_2(T_3))$ obtained for $100,000$ samples from stable distributions with parameters $\alpha$ and $\beta$ chosen uniformly at random from the intervals $(1,2)$ and $(0,1)$ respectively.  The values in Table~\ref{T:5.03} are shown plotted with squares. }
\label{fig:Fig5}
\end{figure}

\begin{figure}[!ht]
\centering
\includegraphics[width=4in]{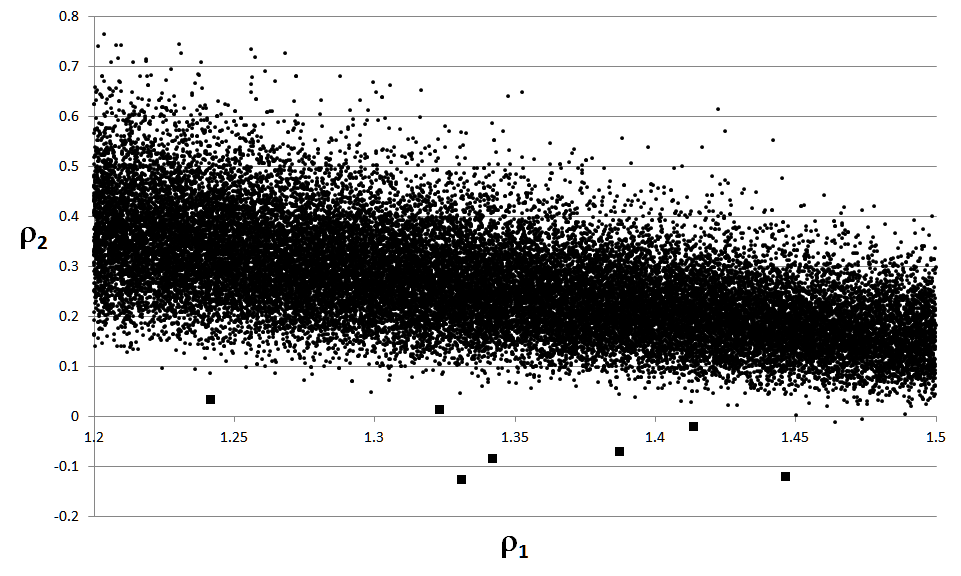}
\caption{Expanded view of Figure~\ref{fig:Fig5}}
\label{fig:Fig6}
\end{figure}

\section{Conclusions and Future Work}\label{S:conclusion}

As we have seen, the statistics $\rho_1$ and $\rho_2$ can be used as the basis for goodness of fit test for financial returns. Much more research needs to be done to better understand the role of the slide statistics in characterizing financial data. In particular, we would like to understand the relationship between the sign of $\rho_2$ and the behaviour of financial markets. The simulations we have described suggest that $\rho_1$ and $\rho_2$ can be used to distinguish between probability distributions and are capapble of detecting dimensional information. The values of $\rho_n$ we obtained through simulation point to a larger theory of these statistics which is currently being developed.  At present however, the $\rho_n$ are quite mysterious and much work will need to be done to understand all they are telling us about sets of points in metric spaces, random variables or point processes in general.  

We gave examples of point processes on the Cantor set and the Sierpinski triangle for which $1/\rho_1$ converged to the dimension of the fractal. We would like to understand when this occurs in general and also the relationship between $1/\rho_1$ and the usual definitions of dimension.  More generally, we would like to know when a point process is tangible in the sense of  Definition~\ref{D:5.04}. When a process is tangible it is possible to use the expression $\sqrt[n]{\frac{(-1)^{n+1}(n-1)!(n-1)\zeta(n)}{\rho_n(U)}}$ as a statistic for estimating the dimension. We gave examples in which it converged to the dimension faster for $n=2$ than for $n=1$ and we would like to know what happens for larger values of $n$.

In terms of calculations, we need to prove Conjecture~\ref{CJ:4.01} concerning the calculation of $\rho_2(U)$ and we need formulas for $\rho_n(U)$ for $n>2$.  Given the complexity of our conjecture for $\rho_2(U)$, the formulas for larger values of $n$ are likely to be very complicated. The convergence of the $\rho_n(U)$ for real-valued random variables can sometimes be improved by using the distances between consecutive points rather than the distances to nearest neighbours. We would like to have a better understanding of this situation and also to know if there are any higher dimensional analogues.

In defining the slide statistics, we used the functions $(f(x))^t$ which can be thought of as a continuous deformation of $f(x)$ at $t=1$ into the constant function $1$ at $t=0$.  We can achieve the same effect using the functions $tf(x)+(1-t)$ and calculate the derivatives corresponding to the slide derivatives in Section~\ref{S:calculus}.  It turns out that only the first of these statistics is interesting and while it is easier to calculate than $\rho_1$ it often doesn't converge.

\end{document}